\documentclass[12pt]{amsart}

\usepackage{amsmath,amssymb,amsthm,enumerate}
\usepackage[margin=1.5in]{geometry}

\DeclareMathOperator{\Tr}{Tr}

\newcommand{\C}{\mathbb{C}}

\newcommand{\N}{\mathbb{N}}
\newcommand{\Z}{\mathbb{Z}}
\newcommand{\Q}{\mathbb{Q}}
\newcommand{\R}{\mathbb{R}}

\newcommand{\F}{\mathbb{F}}

\newcommand{\E}{\mathbb{E}}

\newcommand{\ve}{\varepsilon}

\newcommand{\1}{^{-1}}

\newcommand{\f}[2]{\frac{#1}{#2}}

\newtheorem{theorem}{Theorem}[section]

\newtheorem{corollary}[theorem]{Corollary}
\newtheorem{proposition}[theorem]{Proposition}
\newtheorem{lemma}[theorem]{Lemma}

\numberwithin{theorem}{section}

\title{On the polynomial Szemer\'edi theorem in finite fields}
\author{Sarah Peluse}
\address{Department of Mathematics, Stanford University, Stanford, California 94305}
\email{speluse@stanford.edu}

\begin{document}

\begin{abstract}
Let $P_1,\dots,P_m\in\Z[y]$ be any linearly independent polynomials with zero constant term. We show that there exists $\gamma>0$ such that any subset of $\F_q$ of size at least $q^{1-\gamma}$ contains a nontrivial polynomial progression $x,x+P_1(y),\dots,x+P_m(y)$, provided the characteristic of $\F_q$ is large enough.
\end{abstract}

\maketitle

\section{Introduction}
For $P_1,\dots,P_m\in\Z[y]$ and $S$ equal to either $[N]:=\{1,\dots,N\}$ or a finite field $\F_q$, define $r_{P_1,\dots,P_m}(S)$ to be the size of the largest subset of $S$ that does not contain a progression of the form $x,x+P_1(y),\dots,x+P_m(y)$ with $y\neq 0$. Szemer\'edi's Theorem~\cite{Sz} states that
\begin{equation}\label{szemeredi}
r_{y,2y,\dots,(k-1)y}([N])=o_{k}(N),
\end{equation}
which is equivalent (by a standard compactness argument) to saying that any subset of the integers of positive upper density contains a nontrivial (i.e. with common difference nonzero) $k$-term arithmetic progression $x,x+y,\dots,x+(k-1)y$.

The bound in~(\ref{szemeredi}) does not hold when $y,\dots,(k-1)y$ are replaced by arbitrary integer polynomials. For example, the set $3\N$ contains no progression of the form $x,x+(y^2+1)$, since $y^2+1$ is never divisible by $3$ when $y$ is an integer. However, if we remove the possibility of local obstructions by requiring that $P_1(0)=\dots=P_m(0)=0$, then such a bound does hold. Bergelson and Leibman~\cite{BL} proved that if $P_1(0)=\dots=P_m(0)=0$, then
\[
r_{P_1,\dots,P_m}([N])=o_{P_1,\dots,P_m}(N).
\]

While Gowers~\cite{G1}~\cite{G2} has shown that
\[
r_{y,2y,\dots,(k-1)y}([N])\ll_k\f{N}{(\log\log{N})^{c_k}}
\]
for all $k$, no quantitative bounds are known for the $o_{P_1,\dots,P_m}(N)$ term in Bergelson and Leibman's theorem in general. Aside from when $P_1,\dots,P_m$ are linear, quantitative bounds are known in only two other special cases: when $m=1$ by work of S\'ark\"ozy~\cite{S}~\cite{S2}, Balog, Pelik\'an, Pintz, and Szemer\'edi~\cite{BPPS}, Slijecp\u{c}evi\'c~\cite{Sl}, and Lucier~\cite{L}, and when $P_1,\dots,P_m$ are all homogeneous of the same degree by work of Prendiville~\cite{Pr}. 

Clearly any bounds for $r_{P_1,\dots,P_m}([p])$ automatically hold for $r_{P_1,\dots,P_m}(\F_p)$, but we know even more than this in the finite field setting. Bourgain and Chang~\cite{BC} were the first to consider the problem of bounding $r_{P_1,\dots,P_m}(\F_q)$. They showed that
\[
r_{y,y^2}(\F_p)\ll p^{1-1/15},
\]
and, further, that
\begin{equation}\label{bc}
\#\{(x,y)\in\F_p^2:x,x+y,x+y^2\in A\}=\f{|A|^3}{p}+O(|A|^{3/2}p^{2/5}).
\end{equation}
for any $A\subset\F_p$. Thus, any subset of $\F_p$ of density at least $p^{-1/15+\ve}$ contains very close to the expected number of progressions $x,x+y,x+y^2$ in a random set of the same density.

Bourgain and Chang's proof was quite specific to the progression $x,x+y,x+y^2$, and relied on the explicit evaluation of quadratic Gauss sums. Using a different argument, the author showed in~\cite{P} that a result like Bourgain and Chang's holds when $y$ and $y^2$ are replaced by any two linearly independent polynomials $P_1$ and $P_2$ with $P_1(0)=P_2(0)=0$. The main result of~\cite{P} is that
\begin{equation}\label{p1p2}
\#\{(x,y)\in\F_q^2:x,x+P_1(y),x+P_2(y)\in A\}=\f{|A|^3}{q}+O_{P_1,P_2}(|A|^{3/2}q^{7/16})
\end{equation}
for any $A\subset\F_q$ whenever the characteristic of $\F_q$ is sufficiently large, so that $r_{P_1,P_2}(\F_q)\ll_{P_1,P_2}q^{1-1/24}$.

Note that the exponent of $q$ in the error term of~(\ref{p1p2}) is larger than in~(\ref{bc}), so that the argument in~\cite{P} does not quantitatively recover the result in~\cite{BC}. However, this exponent of $q$ does not depend at all on $P_1$ or $P_2$, so the bound $r_{P_1,P_2}(\F_q)\ll_{P_1,P_2}q^{1-1/24}$ is stronger than what can possibly hold in the integer setting when at least one of $P_1$ or $P_2$ has degree at least $25$.

Dong, Li, and Sawin~\cite{DLS} later improved the error term in~(\ref{p1p2}), showing that
\begin{equation}\label{p1p22}
\#\{(x,y)\in\F_q^2:x,x+P_1(y),x+P_2(y)\in A\}=\f{|A|^3}{q}+O_{P_1,P_2}(|A|^{3/2}q^{3/8})
\end{equation}
whenever $A\subset \F_q$ and the characteristic of $\F_q$ is sufficiently large. This also improves on the error term in Bourgain and Chang's result.

The arguments in~\cite{BC},~\cite{P}, and~\cite{DLS} break down when one tries to use them to study progressions of length longer than three. Currently no results are known for progressions of length at least four when $P_1,\dots,P_m$ are not all of the special form $P_i(y)=a_iy^k$ for a fixed $k\in\N$, which is covered by Prendiville's work~\cite{Pr}.

In this paper, we prove a power-saving bound for $r_{P_1,\dots,P_m}(\F_q)$ for any $P_1,\dots,P_m\in\Z[y]$ that are linearly independent and satisfy $P_1(0)=\dots=P_m(0)=0$, provided the characteristic of $\F_q$ is large enough. Let $\Z[y]_0$ denote the subset of $\Z[y]$ consisting of polynomials with zero constant term.
\begin{theorem}\label{main2}
Let $P_1,\dots,P_m\in\Z[y]_0$ be linearly independent over $\Q$. There exist $c,\gamma>0$ such that if the characteristic of $\F_q$ is at least $c$, then
\[
r_{P_1,\dots,P_m}(\F_q)\ll_{P_1,\dots,P_m}q^{1-\gamma}
\]
and
\[
\#\{(x,y)\in\F_q:x,x+P_1(y),\dots,x+P_m(y)\in A\}=\f{|A|^{m+1}}{q^{m-1}}+O_{P_1,\dots,P_m}(q^{2-(m+1)\gamma})
\]
for every $A\subset\F_q$.
\end{theorem}

While the power saving exponent of $q$ in~(\ref{p1p2}) and~(\ref{p1p22}) is independent of the polynomials $P_1$ and $P_2$, the power saving exponent in Theorem~\ref{main2} depends on $P_1,\dots,P_m$. The dependence is extremely poor, so we do not keep track of it. We also remark that while the earlier papers~\cite{P} and~\cite{DLS} both rely on a decent amount of algebraic geometry machinery, the proof of Theorem~\ref{main2} only requires the Weil bound for curves.

We now briefly describe the proof of Theorem~\ref{main2}. Note that if we can bound the average
\[
\Lambda_{P_1,\dots,P_m}(f_0,\dots,f_m):=\E_{x,y\in\F_q}f_0(x)f_1(x+P_1(y))\cdots f_m(x+P_m(y))
\]
by a negative power of $q$ whenever $\|f_i\|_\infty\leq 1$ for $i=0,\dots,m$ and some $f_i$ has mean zero, then Theorem~\ref{main2} follows easily. Indeed, if $A\subset\F_q$, then the number of progressions $x,x+P_1(y),\dots,x+P_m(y)$ in $A$ equals $q^2\Lambda(1_A,\dots,1_A)$. Now, writing $1_A=f_A+\alpha$ with $\alpha=|A|/q$, we see that $q^2\Lambda(1_A,\dots,1_A)$ equals $q^2\alpha^{m+1}=|A|^{m+1}/q^{m-1}$ plus $2^{m+1}-1$ other terms of the form $q^2\Lambda(f_0,\dots,f_m)$ with at least one $f_i$ equaling $f_A$.

We will prove such a bound on $\Lambda_{P_1,\dots,P_m}(f_0,\dots,f_m)$ by induction on $m$. When $m=1$, this is a simple consequence of the Weil bound. When $m>1$, the proof is no longer so simple. We do know, in general, that a bound of the form
\begin{equation}\label{Us}
|\Lambda_{P_1,\dots,P_m}(f_0,\dots,f_m)|\leq \min_i\|f_i\|_{U^s}^\beta+O(q^{-\beta})
\end{equation}
holds for some $\beta>0$ and $s\in\N$. Here $\|\cdot\|_{U^s}$ is the Gowers $U^s$-norm on functions $f:\F_q\to\C$, whose definition we will recall in Section~\ref{prelims}. If $s=1$, then $\|f\|_{U^s}=|\E_xf(x)|$, in which case certainly $|\Lambda_{P_1,\dots,P_m}(f_0,\dots,f_m)|\ll q^{-\beta}$ whenever some $f_i$ has mean zero. The key idea of the proof is that, if $s>1$, then one can use the bound for progressions of length $m-1$ to deduce a bound similar to~(\ref{Us}), but involving the $U^{s-1}$-norm instead of the $U^s$-norm. Carrying this out $s-1$ times leads to a bound in terms of the $U^1$-norm, and thus of the form $\Lambda_{P_1,\dots,P_m}(f_0,\dots,f_m)\ll q^{-\gamma}$ for some $\gamma>0$.

This paper is organized as follows. In Section~\ref{prelims}, we set notation, recall some standard definitions, and prove a couple of preliminary results needed in the proof of Theorem~\ref{main2}. In Section~\ref{m=1}, we prove Theorem~\ref{main2} (or more precisely, Theorem~\ref{main}) when $m=1$. In Section~\ref{pf}, we describe the inductive step in the proof and show how Theorem~\ref{main2} follows from Lemma~\ref{ind}. We then prove this key lemma in Section~\ref{lemmas}.

\section*{Acknowledgments}
The author thanks Ben Green, Kannan Soundararajan, Julia Wolf, and the anonymous referees for helpful comments on earlier versions of this paper.

The author is supported by the National Science Foundation Graduate Research Fellowship Program under Grant No. DGE-114747 and by the Stanford University Mayfield Graduate Fellowship.

\section{Preliminaries}\label{prelims}
\subsection{Definitions and notation} For every finite set $S$ and $f:S\to\C$, we denote the average of $f$ over $S$ by $\E_{x\in S}f(x):=\f{1}{|S|}\sum_{x\in S}f(x)$. When averaging over $\F_q$, we will sometimes write $\E_x$ instead of $\E_{x\in\F_q}$.

We say that a complex-valued function $f$ is \textit{$1$-bounded} if $\|f\|_{L^\infty}\leq 1$ and that an $m$-tuple of complex-valued functions $(f_1,\dots,f_m)$ is \textit{$1$-bounded} if each of its components $f_i$ is $1$-bounded.

We normalize the $L^p$-norms on $\F_q$ by setting $\|f\|_{L^p}^p:=\E_x|f(x)|^p$, and also set $\langle f,g\rangle:=\E_xf(x)\overline{g(x)}$ for any two $f,g:\F_q\to\C$. If $\|\cdot\|$ is any norm on the $\C$-vector space of functions $f:\F_q\to\C$, its \textit{dual norm} $\|\cdot\|^*$ is defined by
\[
\|f\|^{*}:=\sup_{\|g\|=1}|\langle f,g\rangle|.
\]

For any $f:\F_q\to\C$ and $h\in\F_q$, define $\Delta_hf:\F_q\to\C$ by
\[
\Delta_hf(x):=f(x+h)\overline{f(x)}
\]
for $x\in\F_q$. Also define, for every $h_1,\dots,h_s\in\F_q$, the function $\Delta_{h_1,\dots,h_s}f:\F_q\to\C$ by
\[
\Delta_{h_1,\dots,h_s}f(x)=(\Delta_{h_1}\cdots\Delta_{h_s}f)(x)
\]
for $x\in\F_q$. Note that if $h,k\in\F_q$, then
\[
\Delta_{h}\Delta_{k}f(x)=f(x+h+k)\overline{f(x+h)f(x+k)}f(x)=\Delta_k\Delta_hf(x),
\]
so the ordering of $h_1,\dots,h_s$ in the definition of $\Delta_{h_1,\dots,h_s}f$ does not matter.

For a function $f:\F_q^2\to\C$ of two variables, we define $\Delta_{h_1,\dots,h_s}^{(1)}f:\F_q^2\to\C$ by applying $\Delta_{h_1,\dots,h_s}$ in the first variable of $f$:
\[
\Delta_{h}^{(1)}f(x,y):=f(x+h,y)\overline{f(x,y)}
\]
and
\[
\Delta_{h_1,\dots,h_s}^{(1)}f(x,y):=(\Delta_{h_1}^{(1)}\cdots\Delta_{h_s}^{(1)}f)(x,y).
\]

Now, for any $s\geq 1$, we define the \textit{Gowers $U^s$-norm} $\|\cdot\|_{U^s}$ (which is only a seminorm when $s=1$) by
\[
\|f\|_{U^s}^{2^s}:=\E_{x,h_1,\dots,h_s\in\F_q}\Delta_{h_1,\dots,h_s}f(x)
\]
for $f:\F_q\to\C$. These norms satisfy $\|f\|_{U^s}\leq\|f\|_{U^{s+1}}$ for every $s\geq 1$. The $U^1$-norm of $f$ equals $|\E_xf(x)|$, and the $U^2$-norm of $f$ equals the $\ell^4$-norm of the Fourier transform $\hat{f}(\psi):=\langle f,\psi\rangle$:
\begin{equation}\label{u2}
\|f\|_{U^2}^4=\sum_{\psi\in\widehat{\F}_q}|\hat{f}(\psi)|^4,
\end{equation}
where $\widehat{\F}_q$ denotes the set of additive characters of $\F_q$. One reference for these and other basic properties of Gowers norms is Section 1.3.3 of~\cite{T}.

\subsection{Counting progressions}
Let $m_1\geq 1$, $m_2\geq 0$, and $P_1,\dots,P_{m_1},Q_1,\dots,$ $Q_{m_2}\in\Z[y]$. For every $F=(f_0,\dots,f_{m_1})$ and $G=(g_1,\dots,g_{m_2})$ with $f_i,g_j:\F_q\to\C$ for $0\leq i\leq m_1$ and $1\leq j\leq m_2$, define
\[
\Lambda_{P_1,\dots,P_{m_1}}^{Q_1,\dots,Q_{m_2}}(F;G):=\E_{x,y}f_0(x)\prod_{i=1}^{m_1}f_i(x+P_i(y))\prod_{j=1}^{m_2}g_j(Q_j(y)).
\]
Even though Theorem~\ref{main2} only concerns $\Lambda_{P_1,\dots,P_m}$, we will need to consider these more general averages involving the extra factor $\prod_{j=1}^{m_2}g_j(Q_j(y))$ in order to run an induction argument.

As mentioned earlier, for any $A\subset\F_q$, the quantity $\Lambda_{P_1,\dots,P_{m}}(1_A,\dots,1_A)$ is the normalized count of the number of progressions $x,x+P_1(y),\dots,x+P_m(y)$ in $A$:
\[
\Lambda_{P_1,\dots,P_{m}}(1_A,\dots,1_A)=\f{\#\{(x,y)\in\F_q:x,x+P_1(y),\dots,x+P_m(y)\in A\}}{q^2}.
\]
Theorem~\ref{main2} will thus follow by setting $m_1:=m$, $m_2:=0$, and $f_i:=1_A$ for $i=0,\dots,m_1$ in Theorem~\ref{main}:
\begin{theorem}\label{main}
Let $m_1\geq 1$ and $m_2\geq 0$ and let $P_1,\dots,P_{m_1},Q_1,\dots,Q_{m_2}\in\Z[y]_0$ be linearly independent over $\Q$. There exist $c,\gamma>0$ such that if the characteristic of $\F_q$ is at least $c$, then
\[
\Lambda_{P_1,\dots,P_{m_1}}^{Q_1,\dots,Q_{m_2}}(F;\Psi)=1_{\Psi=1}\prod_{i=0}^{m_1}\E_xf_i(x)+O_{P_1,\dots,P_{m_1}, Q_1,\dots,Q_{m_2}}(q^{-\gamma})
\]
whenever $F=(f_0,\dots,f_{m_1})$ is $1$-bounded and $\Psi\in(\widehat{\F}_q)^{m_2}$.
\end{theorem}
By $1_{\Psi=1}$ here, we mean the quantity that equals $1$ if every component of $\Psi$ is the trivial character and equals $0$ otherwise.

As mentioned in the introduction, the starting point for the proof of Theorem~\ref{main} is a bound for $|\Lambda_{P_1,\dots,P_{m_1}}^{Q_1,\dots,Q_{m_2}}(F;\Psi)|$ in terms of a $U^s$-norm. In order to state it, we will need a definition.

For any finite collection of polynomials $P_1,\dots,P_m\in\F_q[y]$, define their \textit{degree sequence} to be the vector $v=(v_i)_{i=1}^\infty\in(\N\cup\{0\})^\N$ with
\[
v_i:=\#\text{ of distinct leading terms of }P_1,\dots,P_m\text{ of degree }i.
\]
For example, the degree sequence of $y,2y,y^2,y^2+3y,y^5$ is $(2,1,0,0,1,0,\dots)$.

By the same argument that appeared in~\cite{Pr}, which uses the PET induction scheme introduced by Bergelson and Leibman in~\cite{BL}, we have the following proposition.
\begin{proposition}\label{PET1}
Let $P_1,\dots,P_m\in\F_q[y]$. There exist $1\geq \beta>0$ and $s\in\N$ depending only on the degree sequence $v$ of $P_1,\dots,P_m$ such that
\[
|\Lambda_{P_1,\dots,P_m}(F)|\leq\min_{i}\|f_i\|_{U^s}^\beta+O_{v}(q^{-\beta})
\]
for every $1$-bounded $F=(f_0,\dots,f_{m})$.
\end{proposition}
This can be proven by carrying out the argument in Sections 3--5 of~\cite{Pr} almost word-for-word, but in the finite field setting instead of the integer setting. In fact, the proof in finite fields is even simpler than this, since the variables in the definition of $\Lambda_{P_1,\dots,P_m}$ range over all of $\F_q$ instead of over intervals of vastly different sizes as they do in~\cite{Pr}.

It is easy to deduce from Proposition~\ref{PET1} a similar bound for $\Lambda_{P_1,\dots,P_{m_1}}^{Q_1,\dots,Q_{m_2}}(F;\Psi)$ whenever $\Psi\in(\widehat{\F}_q)^{m_2}$.
\begin{proposition}\label{PET2}
Let $P_1,\dots,P_{m_1},Q_1,\dots,Q_{m_2}\in\F_q[y]$. There exist $1\geq\beta>0$ and $s\in\N$ depending only on the degree sequences $v^{(1)},v^{(2)}$ of $P_1,\dots,P_{m_1},$ $Q_1,\dots,Q_{m_2}$ and $P_1,\dots,P_{m_1},Q_1+P_{m_1},\dots,Q_{m_2}+P_{m_1}$ such that
\[
|\Lambda_{P_1,\dots,P_{m_1}}^{Q_1,\dots,Q_{m_2}}(F;\Psi)|\leq\min_i\|f_i\|_{U^s}^\beta+O_{v^{(1)},v^{(2)}}(q^{-\beta})
\]
for every $1$-bounded $F=(f_0,\dots,f_{m_1})$ and $\Psi\in(\widehat{\F}_q)^{m_2}$.
\end{proposition}
\begin{proof}
Note that
\begin{align*}
\Lambda_{P_1,\dots,P_{m_1}}^{Q_1,\dots,Q_{m_2}}(F;\Psi) &= \E_{x,y}f_0(x)\prod_{i=1}^{m_1}f_i(x+P_i(y))\prod_{i=1}^{m_2}\overline{\psi_j(x)}\psi_j(x+Q_j(y)) \\
&= \Lambda_{P_1,\dots,P_{m_1}, Q_1,\dots,Q_{m_2}}(f_0',f_1,\dots,f_{m_1},\psi_1,\dots,\psi_{m_2}),
\end{align*}
where $f_0':=f_0\prod_{j=1}^{m_2}\overline{\psi_j}$. Thus, since all additive characters are $1$-bounded and $f_0'$ is also $1$-bounded if $f_0$ is, we have by Proposition~\ref{PET1} that
\[
|\Lambda_{P_1,\dots,P_{m_1}}^{Q_1,\dots,Q_{m_2}}(F;\Psi)|\leq\min_{i\geq 1}\|f_i\|_{U^{s_1}}^{\beta_1}+O_{v^{(1)}}(q^{-\beta_1})
\]
for some $1\geq\beta_1>0$ and $s_1\in\N$ depending only on $v^{(1)}$.

Similarly,
\[
\Lambda_{P_1,\dots,P_{m_1}}^{Q_1,\dots,Q_{m_2}}(F;\Psi)=\Lambda_{P_1,\dots,P_{m_1},Q_1+P_{m_1},\dots,Q_{m_2}+P_{m_1}}(f_0,\dots,f_{m_1-1},f_{m_1}',\psi_1,\dots,\psi_{m_2}),
\]
where $f_{m_1}':=f_{m_1}\prod_{j=1}^{m_2}\overline{\psi_j}$. Thus, by Proposition~\ref{PET1} again, we have that
\[
|\Lambda_{P_1,\dots,P_{m_1}}^{Q_1,\dots,Q_{m_2}}(F;\Psi)|\leq\min_{i\leq m_1-1}\|f_i\|_{U^{s_2}}^{\beta_2}+O_{v^{(2)}}(q^{-\beta_2})
\]
for some $1\geq \beta_2>0$ and $s_2\in\N$ depending only on $v^{(2)}$.

Since $F$ is $1$-bounded, we have that $\|f_i\|_{U^s}\leq 1$ for all $i=0,\dots,m_1$. Also, recall that $\|f\|_{U^s}\leq\|f\|_{U^{s+1}}$ for all $s\geq 1$. Thus, by setting $\beta=\min(\beta_1,\beta_2)$ and $s=\max(s_1,s_2)$, we have
\[
|\Lambda_{P_1,\dots,P_{m_1}}^{Q_1,\dots,Q_{m_2}}(F;\Psi)|\leq\min_{i}\|f_i\|_{U^{s}}^{\beta}+O_{v^{(1)},v^{(2)}}(q^{-\beta}).
\]
\end{proof}

\subsection{Decomposing functions} The key idea of the proof of Theorem~\ref{main} is that one can turn a bound for $\Lambda_{P_1,\dots,P_{m_1}}^{Q_1,\dots,Q_{m_2}}(F;\Psi)$ in terms of the $U^s$-norm into a bound in terms of the $U^{s-1}$-norm, provided one has shown that the conclusion of Theorem~\ref{main} holds when the pair $(m_1,m_2)$ is replaced by $(m_1-1,m_2+1)$. To carry out this step of the proof, we first decompose $f_0$ as
\begin{equation}\label{decomp}
f_0=f_a+f_b+f_c,
\end{equation}
where $\|f_b\|_{L^1}$ and $\|f_c\|_{U^s}$ are small and $\|f_a\|_{U^s}^*$ and $\|f_c\|_\infty$ are not too large. Inserting $f_a+f_b+f_c$ in place of $f_0$ and using multilinearity, we can bound $\Lambda_{P_1,\dots,P_{m_1}}^{Q_1,\dots,Q_{m_2}}(F;\Phi)$ in terms of the $U^s$-norm of a dual function, plus some small error depending on the size of $\|f_c\|_{U^s}$ and $\|f_b\|_{L^1}$.

Our final task in this section is to prove that such a decomposition always exists. To do this, we will use a technique due to Gowers. In~\cite{G3}, Gowers describes a general method for proving decomposition results for functions using the hyperplane-separation version of the Hahn-Banach theorem. This method was used in~\cite{G3} to give a new proof of the transference principle, and was also used by Gowers and Wolf~\cite{GW2}~\cite{GW1}~\cite{GW3} in work on the true complexity of systems of linear forms.

To prove our decomposition result, we will use the following corollary of the Hahn-Banach theorem from~\cite{G3}:
\begin{corollary}[Corollary 3.2 of~\cite{G3}]\label{sep}
Let $K_1,\dots,K_r$ be closed convex subsets of $\R^n$ that all contain $0$, and let $c_1,\dots,c_r>0$. Suppose that $f\in\R^n$ cannot be written as
\[
f=f_1+\dots+f_r
\]
with $f_i\in c_iK_i$ for $i=1,\dots,r$. Then there exists $\phi\in\R^n$ such that $\langle f,\phi\rangle>1$ and $\langle g_i,\phi\rangle\leq c_i\1$ for every $g_i\in K_i$, $i=1,\dots,r$.
\end{corollary}

We will also need the following special case of Lemma 3.4 from~\cite{G3}:
\begin{lemma}[Special case of Lemma 3.4 of~\cite{G3}]\label{l1}
Let $\|\cdot\|_1$ and $\|\cdot\|_2$ be norms on $\R^n$, and define another norm on $\R^n$ by
\[
\|f\|:=\inf\{\|f_1\|_1+\|f_2\|_2:f_1+f_2=f\}.
\]
Then $\|g\|^*=\max(\|g\|_1^*,\|g\|_2^*)$.
\end{lemma}

Now we can prove that the decomposition in~(\ref{decomp}) exists.
\begin{proposition}\label{hb}
Let $\|\cdot\|$ be any norm on the $\C$-vector space of complex-valued functions on $\F_q$, and let $\delta_1,\delta_2,\delta_3,\delta_4>0$. Suppose that $f:\F_q\to\C$ with $\|f\|_{L^2}\leq 1$. If $q^{\delta_2-\delta_3}+q^{\delta_4-\delta_1}\leq 1/2$, then there exist $f_a,f_b,f_c:\F_q\to\C$ such that
\[
f=f_a+f_b+f_c,
\]
$\|f_a\|^*\leq q^{\delta_1}$, $\|f_b\|_{L^1}\leq q^{-\delta_2}$, $\|f_c\|_{L^\infty}\leq q^{\delta_3}$, and $\|f_c\|\leq q^{-\delta_4}$.
\end{proposition}
\begin{proof}
Note that it suffices to prove the result for real-valued functions, for we can write $f=g+ih$ where $g$ and $h$ are real-valued and $\|f\|_{L^2}^2=\|g\|_{L^2}^2+\|h\|_{L^2}^2$. So assume for the remainder of this proof that $f$ is real-valued.

Suppose by way of contradiction that no such decomposition of $f$ exists. Define a norm $\|\cdot\|'$ on the $\R$-vector space of functions $\F_q\to\R$ by $\|f\|':=\max(q^{-\delta_3}\|f\|_{L^\infty},q^{\delta_4}\|f\|)$. Note that this vector space is isomorphic to $\R^q$. We apply Corollary~\ref{sep} with $f$ and the subsets
\[
K_1:=\{g:\F_q\to\R:\|g\|^*\leq q^{\delta_1}\},
\]
\[
K_2:=\{g:\F_q\to\R:\|g\|_{L^1}\leq q^{-\delta_2}\},
\]
and
\[
K_3:=\{g:\F_q\to\R:\|g\|'\leq 1\},
\]
which are all closed, convex, and contain $0$ since they are each the scaled closed unit ball of some norm. So, there exists $\phi:\F_q\to\R$ such that $\langle f,\phi\rangle>1$ and $\langle g_i,\phi\rangle\leq 1$ for every $g_i\in K_i$, $i=1,2,3$.

Since $\|\cdot\|^{**}=\|\cdot\|$ and $\langle g,\phi\rangle\leq 1$ whenever $\|g\|^*\leq q^{\delta_1}$, we have that $\|\phi\|\leq q^{-\delta_1}$. Similarly, since $\|\cdot\|_{L^\infty}^*=\|\cdot\|_{L^1}$ and $\langle g,\phi\rangle\leq 1$ whenever $\|g\|_{L^1}\leq q^{-\delta_2}$, we have that $\|\phi\|_{L^\infty}\leq q^{\delta_2}$. For the same reason, we also have $\|\phi\|'^*\leq 1$, which by Lemma~\ref{l1} implies that
\[
\inf\{q^{\delta_3}\|\phi_1\|_{L^1}+q^{-\delta_4}\|\phi_2\|^*:\phi_1+\phi_2=\phi\}\leq 1.
\]
Thus, there exist $\phi_1,\phi_2:\F_q\to\R$ such that $\phi=\phi_1+\phi_2$ and $q^{\delta_3}\|\phi_1\|_{L^1}+q^{-\delta_4}\|\phi_2\|^*\leq 2$, which implies that $\|\phi_1\|_{L^1}\leq 2q^{-\delta_3}$ and $\|\phi_2\|^*\leq 2q^{\delta_4}$.

Now, $\|\phi\|_{L^2}^2=\langle \phi,\phi\rangle=\langle \phi,\phi_1\rangle+\langle \phi,\phi_2\rangle$, and by the above, we have that
\[
|\langle \phi,\phi_1\rangle|\leq \|\phi\|_{L^\infty}\|\phi_1\|_{L^1}\leq 2q^{\delta_2-\delta_3}
\]
and, similarly, that
\[
|\langle \phi,\phi_2\rangle|\leq \|\phi\|\|\phi_2\|^*\leq 2q^{\delta_4-\delta_1}.
\]
Thus, $\|\phi\|_{L^2}^2\leq 2(q^{\delta_2-\delta_3}+q^{\delta_4-\delta_1})$. However, we also have that $1<\langle f,\psi\rangle\leq\|f\|_{L^2}\|\phi\|_{L^2}$ by Cauchy--Schwarz, so that $\|\phi\|_{L^2}^2>1$ since $\|f\|_{L^2}\leq 1$. This gives us a contradiction whenever $q^{\delta_2-\delta_3}+q^{\delta_4-\delta_1}\leq 1/2$.
\end{proof}

\section{Proof of Theorem~\ref{main} when $m_1=1$}\label{m=1}
As mentioned in the introduction, we will prove Theorem~\ref{main} by induction on $m_1$. In this section, we show that the conclusion of Theorem~\ref{main} holds when $m_1=1$.

The following is a simple consequence of the Weil bound and the classification of additive characters of $\F_q$, both of whose proofs can be found in~\cite{K}.
\begin{lemma}\label{weil}
Let $P_1,\dots,P_m\in\Z[y]_0$ be linearly independent over $\Q$. There exists $c>0$ such that if the characteristic of $\F_q$ is at least $c$ and $\psi_1,\dots,\psi_m\in\widehat{\F}_q$ are not all trivial, then
\[
\E_{y}\prod_{i=1}^m\psi_i(P_i(y))\ll_{P_1,\dots,P_m}q^{-1/2}.
\]
\end{lemma}
\begin{proof}
Let $c>0$ be large enough so that $P_1,\dots,P_m$ are linearly independent modulo any prime larger than $c$. To see that such a $c$ exists, set $d:=\max_i\deg P_i$, form the $(d+1)\times m$ matrix $M$ of coefficients of $P_1,\dots,P_m$, let $C$ be a nonvanishing $m\times m$ minor of $M$ (which exists by the linear independence assumption on $P_1,\dots,P_m$), and just pick $c$ larger than all primes dividing $C$. Assume that $\F_q$ has characteristic $p\geq c$.

The additive characters of $\F_q$ are exactly the functions $x\mapsto e_p(\Tr_{\F_q/\F_p}(ax))$ for $a\in\F_q$. Since $\psi_j\neq 1$ for some $j=1,\dots,m$, there exist $a_1,\dots,a_m\in\F_q$ with some $a_j\neq 0$ such that $\psi_i(x)=e_p(\Tr_{\F_q/\F_p}(a_ix))$ for each $i=1,\dots,m$. Thus,
\[
\prod_{i=1}^{m}\psi_i(P_i(y))=e_p(\Tr_{\F_q/\F_p}(P(y))),
\]
where $P(y):=\sum_{i=1}^ma_iP_i(y)$, since $\Tr_{\F_q/\F_p}$ is $\F_p$-linear.

The polynomial $P$ is nonconstant since $P_1,\dots,P_m$ are linearly independent and $a_j\neq 0$. Thus, by the Weil bound, we have that
\[
\E_ye_p(\Tr_{\F_q/\F_p}(P(y)))\ll_{\deg{P}}q^{-1/2}\ll_{P_1,\dots,P_m}q^{-1/2},
\]
which completes the proof of the lemma.
\end{proof}
Now we can prove Theorem~\ref{main} in the $m_1=1$ case.
\begin{lemma}\label{base}
Let $m_2\geq 0$ and let $P_1,Q_1,\dots,Q_{m_2}\in\Z[y]_0$ be linearly independent over $\Q$. There exists $c>0$ such that if the characteristic of $\F_q$ is at least $c$, then
\[
|\Lambda_{P_1}^{Q_1,\dots,Q_{m_2}}(F;\Psi)-1_{\Psi=1}\prod_{i=0}^1\E_zf_i(z)|\ll_{P_1,Q_1,\dots,Q_{m_2}}q^{-1/2}
\]
whenever $F=(f_0,f_1)$ is $1$-bounded and $\Psi\in(\widehat{\F}_q)^{m_2}$.
\end{lemma}
\begin{proof}
By Lemma~\ref{weil}, there exists $c>0$ such that
\begin{equation}\label{charsum}
\E_{y}\phi(P_1(y))\prod_{i=1}^{m_2}\psi_{i}(Q_i(y))\ll_{P_1,Q_1,\dots,Q_{m_2}}q^{-1/2}
\end{equation}
whenever the characteristic of $\F_q$ is at least $c$ and $\phi,\psi_1,\dots,\psi_{m_2}\in\widehat{\F}_q$ are not all trivial.

Set $f_1':=f_1-\E_zf_1(z)$ and $F':=(f_0,f_1')$. Since $f_1=\E_zf_1(z)+f_1'$ and $\Lambda_{P_1}^{Q_1,\dots,Q_{m_2}}$ is bilinear, we have
\begin{align*}
\Lambda_{P_1}^{Q_1,\dots,Q_{m_2}}(F;\Psi) &= (\E_{z}f_1(z))\E_{x,y}f_0(x)\prod_{j=1}^{m_2}\psi_j(Q_j(y))+\Lambda_{P_1}^{Q_1,\dots,Q_{m_2}}(F';\Psi) \\
&= \left(\prod_{i=0}^1\E_zf_i(z)\right)\E_y\prod_{j=1}^{m_2}\psi_j(Q_j(y))+\Lambda_{P_1}^{Q_1,\dots,Q_{m_2}}(F';\Psi).
\end{align*}

 Assume that $\F_q$ has characteristic at least $c$. If $\psi_j=1$ for all $j=1,\dots,m_2$, then $\E_{y}\prod_{j=1}^{m_2}\psi_j(Q_j(y))=1$. Otherwise, $\E_y\prod_{j=1}^{m_2}\psi_j(Q_j(y))\ll_{Q_1,\dots,Q_{m_2}}q^{-1/2}$ by~(\ref{charsum}). Thus,
\[
\E_y\prod_{j=1}^{m_2}\psi_j(Q_j(y))=1_{\Psi=1}+O_{Q_1,\dots,Q_{m_2}}(q^{-1/2}),
\]
and since $f_0$ and $f_1$ are $1$-bounded, this implies that
\[
\left(\prod_{i=0}^1\E_zf_i(z)\right)\E_y\prod_{j=1}^{m_2}\psi_j(Q_j(y))=1_{\Psi=1}\prod_{i=0}^1\E_zf_i(z)+O_{Q_1,\dots,Q_{m_2}}(q^{-1/2}).
\]

Now, by Fourier inversion, we have
\begin{align*}
\Lambda_{P_1}^{Q_1,\dots,Q_{m_2}}(F';\Psi) &= \sum_{\eta_0,\eta_1\in\widehat{\F}_q}\widehat{f_0}(\eta_0)\widehat{f'_1}(\eta_1)\left(\E_{x}\eta_0(x)\eta_1(x)\right)\left(\E_y\eta_1(P_1(y))\prod_{j=1}^{m_2}\psi_j(Q_j(y))\right) \\
&= \sum_{1\neq\eta\in\widehat{\F}_q}\widehat{f_0}(\eta)\overline{\widehat{f'_1}(\eta)}\left(\E_y\overline{\eta(P_1(y))}\prod_{j=1}^{m_2}\psi_j(Q_j(y))\right),
\end{align*}
since $f_1'$ has mean zero. By~(\ref{charsum}) again,
\[
\E_y\overline{\eta(P(y))}\prod_{j=1}^{m_2}\psi_j(Q_j(y))\ll_{P_1,Q_1,\dots,Q_{m_2}}q^{-1/2}
\]
whenever $\eta\neq 1$. Since $f_0$ is $1$-bounded and $\|f_1'\|_\infty\leq 2$, we have that $\sum_{1\neq\eta\in\widehat{\F}_q}\widehat{f_0}(\eta)\overline{\widehat{f'_1}(\eta)}\ll 1$ by Cauchy--Schwarz and Parseval's identity. Hence, $|\Lambda_{P_1}^{Q_1,\dots,Q_{m_2}}(F';\Psi)|\ll_{P_1,Q_1,\dots,Q_{m_2}}q^{-1/2}$.
\end{proof}

\section{The inductive step and the proof of Theorem~\ref{main}}\label{pf}

To prove Theorem~\ref{main} in general, we proceed by induction on $m_1$. The idea of the proof is that if
\[
|\Lambda_{P_1,\dots,P_{m_1}}^{Q_1,\dots,Q_{m_2}}(F;\Psi)|\leq\min_i\|f_i\|_{U^s}^\beta+O(q^{-\beta})
\]
for some $1\geq\beta>0$ and $s\in\N$, $s\geq 2$, then we can combine Proposition~\ref{hb} with the conclusion of Theorem~\ref{main} with $(m_1-1,m_2+1)$ in place of $(m_1,m_2)$ to bound $\Lambda_{P_1,\dots,P_{m_1}}^{Q_1,\dots,Q_{m_2}}(F;\Psi)$ in terms of the $U^{s-1}$-norm of the $f_i$'s. We then deduce a bound for $\Lambda_{P_1,\dots,P_{m_1}}^{Q_1,\dots,Q_{m_2}}(F;\Psi)$ in terms of the $U^1$-norm of the $f_i$'s by repeating this $s-2$ more times.

\subsection{A simplified example of the argument}
Suppose, for the sake of illustration, that
\begin{equation}\label{ex}
|\Lambda_{P_1,P_2}(F)|\leq\min_{i}\|f_i\|_{U^2}^{1/4}+q^{-1/4}
\end{equation}
for every $1$-bounded $F=(f_0,f_1,f_2)$. The purpose of this subsection is to give the simplest possible demonstration of how we can turn a bound for $\Lambda$ in terms of the $U^s$-norm into one in terms of the $U^{s-1}$-norm, so that the proof of Theorem~\ref{main} will hopefully be easier to follow. We do not claim that~(\ref{ex}) actually holds.

Let $\delta_1,\delta_2,\delta_3,\delta_4>0$ with $\delta_2<\delta_3$ and $\delta_4<\delta_1$, to be chosen later. Assume that $q$ is large enough so that $q^{\delta_2-\delta_3}+q^{\delta_4-\delta_1}\leq 1/2$. Let $F=(f_0,f_1,f_2)$ be $1$-bounded. By Proposition~\ref{hb}, we can write
\[
f_0=f_a+f_b+f_c
\]
for some $f_a,f_b,f_c:\F_q\to\C$ with $\|f_a\|_{U^2}^*\leq q^{\delta_1}$, $\|f_b\|_{L^1}\leq q^{-\delta_2}$, $\|f_c\|_{L^\infty}\leq q^{\delta_3}$, and $\|f_c\|_{U^2}\leq q^{-\delta_4}$. Then
\[
\Lambda_{P_1,P_2}(f_0,f_1,f_2)=\Lambda_{P_1,P_2}(f_a,f_1,f_2)+\Lambda_{P_1,P_2}(f_b,f_1,f_2)+\Lambda_{P_1,P_2}(f_c,f_1,f_2)
\]
using the trilinearity of $\Lambda_{P_1,P_2}$.

The term $\Lambda_{P_1,P_2}(f_b,f_1,f_2)$ is the simplest to handle. We use the triangle inequality and the $1$-boundedness of $f_1$ and $f_2$ to bound $|\Lambda_{P_1,P_2}(f_b,f_1,f_2)|$ by
\[
\E_{x,y}|f_b(x)||f_1(x+P_1(y))||f_2(x+P_2(y))|\leq \|f_b\|_{L^1}\leq q^{-\delta_2}.
\]

To bound $\Lambda_{P_1,P_2}(f_c,f_1,f_2)$, note that
\[
\Lambda_{P_1,P_2}(f_c,f_1,f_2)=q^{\delta_3}\Lambda_{P_1,P_2}\left(q^{-\delta_3}f_c,f_1,f_2\right)
\]
and that $q^{-\delta_3}f_c$ is $1$-bounded. Since $f_1$ and $f_2$ are $1$-bounded as well, we get from~(\ref{ex}) that
\[
\Lambda_{P_1,P_2}\left(q^{-\delta_3}f_c,f_1,f_2\right)\leq\|q^{-\delta_3}f_c\|_{U^2}^{1/4}+q^{-1/4}\leq q^{-(\delta_3+\delta_4)/4}+q^{-1/4}.
\]
Thus, $\Lambda_{P_1,P_2}(f_c,f_1,f_2)\leq q^{3\delta_3/4-\delta_4/4}+q^{\delta_3-1/4}$.

Finally, to bound $\Lambda_{P_1,P_2}(f_a,f_1,f_2)$, we set
\[
g(x):=\E_{y}f_1(x+P_1(y))f_2(x+P_2(y)),
\]
so that $\Lambda_{P_1,P_2}(f_a,f_1,f_2)=\E_xf_a(x)g(x)$. Since  $|\E_xf_a(x)g(x)|\leq \|f_a\|_{U^2}^*\|g\|_{U^2}$, this implies that $|\Lambda_{P_1,P_2}(f_a,f_1,f_2)|\leq q^{\delta_1}\|g\|_{U^2}$. Now, $\|g\|_{U^2}^2\leq\max_{\psi\in\widehat{\F}_q}|\widehat{g}(\psi)|$ by~(\ref{u2}) because $g$ is $1$-bounded. But for every $\psi\in\widehat{\F}_q$ we have
\begin{align*}
\hat{g}(\psi) &= \E_{x,y}\overline{\psi(x)}f_1(x+P_1(y))f_2(x+P_2(y)) \\
&= \E_{x,y}\overline{\psi(x-P_1(y))}f_1(x)f_2(x+P_2(y)-P_1(y)) \\
&= \E_{x,y}(\overline{\psi}f_1)(x)f_2(x+P_2(y)-P_1(y))\psi(P_1(y)) \\
&= \Lambda_{P_2-P_1}^{P_1}(\overline{\psi} f_1,f_2;\psi),
\end{align*}
which we can estimate using Lemma~\ref{base}, since if $P_1$ and $P_2$ are linearly independent, then so are $P_2-P_1$ and $P_1$. Lemma~\ref{base} tells us that
\begin{equation}\label{exa1}
\Lambda_{P_2-P_1}^{P_1}(\overline{\psi} f_1,f_2;\psi)=1_{\psi=1}(\E_zf_1(z))(\E_z f_2(z))+O_{P_1,P_2}(q^{-1/2})
\end{equation}
whenever the characteristic of $\F_q$ is sufficiently large.

Since $f_1$ and $f_2$ are $1$-bounded, we can bound $|1_{\psi=1}(\E_zf_1(z))(\E_z f_2(z))|$ above by $\min_{i=1,2}\|f_i\|_{U^1}$. This shows that
\[
|\Lambda_{P_1,P_2}(f_a,f_1,f_2)|\leq q^{\delta_1}(\min_{i=1,2}\|f_i\|_{U^1}^{1/2}+O_{P_1,P_2}(q^{-1/4}))
\]
whenever~(\ref{exa1}) holds, since $x_1^{1/2}+x_2^{1/2}>(x_1+x_2)^{1/2}$ for all $x_1,x_2>0$. Thus,
\begin{equation}\label{f2}
|\Lambda_{P_1,P_2}(f_0,f_1,f_2)|\leq q^{\delta_1}\min_{i=1,2}\|f_i\|_{U^1}^{1/2}+O_{P_1,P_2}(q^{\delta_1-1/4}+q^{-\delta_2}+q^{3\delta_3/4-\delta_4/4}+q^{\delta_3-1/4})
\end{equation}
whenever $q$ and the characteristic of $\F_q$ are large enough so that $q^{\delta_2-\delta_3}+q^{\delta_4-\delta_1}\leq 1/2$ and~(\ref{exa1}), respectively, hold.

Now write $f_2':=f_2-\E_zf_2(z)$, so that $\f{1}{2}f_2'$ is $1$-bounded and has mean zero (i.e., $\|\f{1}{2}f_2'\|_{U^1}=0$). Then
\[
\Lambda_{P_1,P_2}(f_0,f_1,f_2)=\Lambda_{P_1}(f_0,f_1)\E_zf_2(z)+2\Lambda_{P_1,P_2}\left(f_0,f_1,\f{1}{2}f_2'\right).
\]
We have, by Lemma~\ref{base}, that
\begin{equation}\label{exa2}
\Lambda_{P_1}(f_0,f_1)=(\E_zf_0(z))(\E_zf_1(z))+O_{P_1}(q^{-1/2})
\end{equation}
whenever the characteristic of $\F_q$ is large enough, and, by~(\ref{f2}), that
\[
\left|\Lambda_{P_1,P_2}\left(f_0,f_1,\f{1}{2}f_2'\right)\right|\ll_{P_1,P_2}q^{\delta_1-1/4}+q^{-\delta_2}+q^{3\delta_3/4-\delta_4/4}+q^{\delta_3-1/4}
\]
whenever $q$ and the characteristic of $\F_q$ are large enough.

In order to bound the above by a negative power of $q$, we must choose $\delta_1,\delta_2,\delta_3,\delta_4>0$ with $\delta_2<\delta_3$, $\delta_4<\delta_1$, $\delta_1<1/4$, $3\delta_3/4<\delta_4/4$, and $\delta_3<1/4$. One simple choice that works is $\delta_1=1/8$, $\delta_2=1/256$, $\delta_3=1/128$, and $\delta_4=1/16$, so that $q^{\delta_1-1/4}+q^{-\delta_2}+q^{3\delta_3/4-\delta_4/4}+q^{\delta_3-1/4}\ll q^{-1/256}$.

We conclude, under the assumption~(\ref{ex}), that if $q$ is large enough so that $q^{-1/256}+q^{-1/16}\leq 1/2$ and the characteristic of $\F_q$ is large enough so that~(\ref{exa1}) and~(\ref{exa2}) hold, then
\[
\Lambda_{P_1,P_2}(f_0,f_1,f_2)=\prod_{i=0}^2\E_zf_i(z)+O_{P_1,P_2}(q^{-1/256}).
\]

\subsection{Proof of Theorem~\ref{main}}
Lemma~\ref{ind} below describes in general a bound for $\Lambda_{P_1,\dots,P_{m_1}}^{Q_1,\dots,Q_{m_2}}(F;\Psi)$ in terms of the $U^{s-1}$-norm that one can derive from any bound in terms of the $U^s$-norm, assuming that the conclusion of Theorem~\ref{main} holds for the pair $(m_1-1,m_2+1)$. The proof of the lemma, which we postpone to the next section, is modeled after the argument given in the previous subsection. However, the argument is not nearly as straightforward if a $U^s$-norm with $s>2$ is involved, which is the typical situation. One can still argue in a similar manner to get an upper bound for $\Lambda_{P_1,\dots,P_m}$ in terms of the $U^s$-norm of the dual function. The key to the remainder of the proof is a lemma that returns us to the $U^2$ situation, which allows us to avoid the use of the $U^s$-inverse theorem when $s>2$.

The statement of Lemma~\ref{ind} is long and involves many parameters, which are necessary to run the induction argument in the proof of Theorem~\ref{main}. Ignoring the parameters, the basic idea of the lemma is that if one has a bound for $\Lambda$ in terms of the $U^s$-norm, then one also has a bound in terms of the $U^{s-1}$-norm, plus some error. We apply the lemma repeatedly to prove Theorem~\ref{main}, and then at the end of the proof select the $\delta_1,\delta_2,\delta_3,$ and $\delta_4$ parameters in each iteration of the lemma so that this error decays polynomially in $1/q$.
\begin{lemma}\label{ind}
Let $m_1\geq 2$ and $m_2\geq 0$. Assume that for all linearly independent $R_1,\dots,$ $R_{m_1-1},S_1,\dots,S_{m_2+1}\in\Z[y]_0$ there exist $c_1,c_2,\gamma>0$ such that
\begin{equation}\label{m1}
\big|\Lambda_{R_1,\dots,R_{m_1-1}}^{S_1,\dots,S_{m_2+1}}(G;\Phi)-1_{\Phi=1}\prod_{i=0}^{m_1-1}\E_xg_i(x)\big|\leq\f{c_2}{q^\gamma}
\end{equation}
for every $1$-bounded $G=(g_0,\dots,g_{m_1-1})$ and $\Phi\in(\widehat{\F}_q)^{m_2+1}$, whenever the characteristic of $\F_q$ is at least $c_1$.

Suppose that $P_1,\dots,P_{m_1},Q_1,\dots,Q_{m_2}\in\Z[y]_0$ are linearly independent and that there exist $b_1,b_2,b_3,b_4>0$ and $s\in\N$, $s\geq 2$, such that
\begin{equation}\label{s}
|\Lambda_{P_1,\dots,P_{m_1}}^{Q_1,\dots,Q_{m_2}}(F;\Psi)|\leq b_1\min_j\|f_j\|_{U^s}^{b_2}+b_3
\end{equation}
for every $1$-bounded $F=(f_0,\dots,f_{m_1})$ and $\Psi\in(\widehat{\F}_q)^{m_2}$, whenever the characteristic of $\F_q$ is at least $b_4$. Then there exist $c_1',c_2',\gamma'>0$ depending only on $P_1,\dots,P_{m_1},Q_1,\dots,Q_{m_2}$ (and not on $b_1,b_2,b_3,b_4,$ or $s$) such that, whenever $F=(f_0,\dots,f_{m_1})$ is $1$-bounded, $\Psi\in(\widehat{\F}_q)^{m_2}$, and the characteristic of $\F_q$ is at least $\max(c_1',b_4)$, we have that $|\Lambda_{P_1,\dots,P_{m_1}}^{Q_1,\dots,Q_{m_2}}(F;\Psi)|$ is bounded above by
\[
q^{\delta_1}\min_{i\geq 0}\|f_i\|_{U^{s-1}}^{2^{1-s}}+q^{\delta_1}\left(\f{c_2'}{q^{\gamma'}}\right)^{2^{2-2s}}+q^{-\delta_2}+q^{(1-b_2)\delta_3-b_2\delta_4}b_1+q^{\delta_3}b_3
\]
for every $\delta_1,\delta_2,\delta_3,\delta_4>0$ satisfying $q^{\delta_2-\delta_3}+q^{\delta_4-\delta_1}\leq 1/2$.
\end{lemma}

We can now prove Theorem~\ref{main}.

\begin{proof}[Proof of Theorem~\ref{main}]
We proceed by induction on $m_1$. Lemma~\ref{base} provides the base case for the induction, so let $m_1\geq 2$ and assume that we have proved the theorem for any pair $(m_1',m_2')$ with $m_1'<m_1$.

First, we have by Proposition~\ref{PET2} that
\begin{equation}\label{start}
|\Lambda_{P_1,\dots,P_{m_1}}^{Q_1,\dots,Q_{m_2}}(F;\Psi)|\leq \min_{i}\|f_i\|_{U^s}^\beta+O_{P_1,\dots,P_{m_1},Q_1,\dots,Q_{m_2}}(q^{-\beta})
\end{equation}
for some $1\geq \beta>0$ and $s\in\N$, $s\geq 2$, depending only on the polynomials $P_1,\dots,P_{m_1},Q_1,\dots,Q_{m_2}$.

Let $\delta_{k}^{(\ell)}>0$, $1\leq k\leq 4$ and $2\leq\ell\leq s$, be parameters to be chosen later that satisfy $\delta_2^{(\ell)}<\delta_3^{(\ell)}$ and $\delta_4^{(\ell)}<\delta_1^{(\ell)}$ for each $\ell$. Starting with the bound~(\ref{start}), we apply Lemma~\ref{ind} repeatedly to bound $|\Lambda_{P_1,\dots,P_{m_1}}^{Q_1,\dots,Q_{m_2}}(F;\Psi)|$, using the parameters $\delta_1^{(\ell)},\delta_2^{(\ell)},\delta_3^{(\ell)},\delta_4^{(\ell)}$ to move from a bound
\[
|\Lambda_{P_1,\dots,P_{m_1}}^{Q_1,\dots,Q_{m_2}}(F;\Psi)|\leq b_1^{(\ell)}\min_{i}\|f_i\|_{U^{\ell}}^{b_2^{(\ell)}}+b_3^{(\ell)}
\]
in terms of the $U^{\ell}$-norm to a bound
\[
|\Lambda_{P_1,\dots,P_{m_1}}^{Q_1,\dots,Q_{m_2}}(F;\Psi)|\leq b_1^{(\ell-1)}\min_i\|f_i\|_{U^{\ell-1}}^{b_2^{(\ell-1)}}+b_3^{(\ell-1)}
\]
in terms of the $U^{\ell-1}$-norm, where $b_1^{(\ell-1)}=2q^{\delta_1^{(\ell)}}$, $b_2^{(\ell-1)}=2^{1-\ell}$, and
\begin{equation}\label{b3}
b_3^{(\ell-1)}\ll_{P_1,\dots,P_{m_1},Q_1,\dots,Q_{m_2}}q^{(1-b_2^{(\ell)})\delta_3^{(\ell)}-b_2^{(\ell)}\delta_4^{(\ell)}}b_1^{(\ell)}+q^{-\delta_2^{(\ell)}}+q^{\delta_1^{(\ell)}-\gamma 2^{2-2\ell}}+q^{\delta_3^{(\ell)}}b_3^{(\ell)}.
\end{equation}
Since $b_1^{(1)}=q^{\delta_1^{(2)}}$ and $b_2^{(1)}=1/2$, this leads to the bound
\[
|\Lambda_{P_1,\dots,P_{m_1}}^{Q_1,\dots,Q_{m_2}}(F;\Psi)|\leq 2q^{\delta_1^{(2)}}\min_{i}\|f_i\|_{U^1}^{1/2}+b_3^{(1)}
\]
whenever the characteristic of $\F_q$ is sufficiently large depending on the polynomials $P_1,\dots,P_{m_1},Q_1,\dots,Q_{m_2}$ and the $\delta_k^{(\ell)}$'s. 

Let $f_{m_1}':=f_{m_1}-\E_xf_{m_1}(x)$, so that $f_{m_1}'$ has mean zero and $\f{1}{2}f_{m_1}'$ is $1$-bounded. Then we have that
\begin{align*}
|\Lambda_{P_1,\dots,P_{m_1}}^{Q_1,\dots,Q_{m_2}}(F;\Psi)-\E_xf_{m_1}(x)\Lambda_{P_1,\dots,P_{m_1-1}}^{Q_1,\dots,Q_{m_2}}(f_0,\dots,f_{m_1-1};\Psi)|&\leq 2|\Lambda_{P_1,\dots,P_{m_1}}^{Q_1,\dots,Q_{m_2}}(f_0,\dots,f_{m_1-1},\f{1}{2}f_{m_1}';\Psi)| \\
&\leq 2b_3^{(1)}
\end{align*}
whenever the characteristic of $\F_q$ is sufficiently large, since $\|\f{1}{2}f_{m_1}'\|_{U^1}=0$. Define the function $\exp_q:\R\to\R$ by $\exp_q(u)=q^u$ for ease of notation.  Applying the bound~(\ref{b3}) recursively we get that
\begin{align*}
b_3^{(1)}\ll_{P_1,\dots,P_{m_1},Q_1,\dots,Q_{m_2}}&b_3^{(s)}\exp_q\left(\sum_{i=0}^{s-2}\delta_3^{(s-i)}\right)+\\
&\sum_{j=0}^{s-2}b_1^{(s-j)}\exp_q\left((1-b_2^{(2-j)})\delta_3^{(s-j)}-b_2^{(s-j)}\delta_4^{(s-j)}+\sum_{i=j+1}^{s-2}\delta_3^{(s-i)}\right)+ \\
&\sum_{j=0}^{s-2}\exp_q\left(-\delta_2^{(s-j)}+\sum_{i=j+1}^{s-2}\delta_3^{(s-i)}\right)+\\
&\sum_{j=0}^{s-2}\exp_q\left(\delta_1^{(s-j)}-\gamma 2^{2-2(s-j)}+\sum_{i=j+1}^{s-2}\delta_3^{(s-i)}\right),
\end{align*}
and thus, using that $b_1^{(s)}=1$, $b_2^{(s)}=\beta$, $b_3^{(s)}\ll_{P_1,\dots,P_{m_1}, Q_1,\dots,Q_{m_2}}q^{-\beta}$, $b_1^{(s-j)}=2q^{\delta_1^{(s-j+1)}}$, and $b_2^{(s-j)}=2^{1-(s-j+1)}$ when $j>0$, that
\begin{align*}
b_3^{(1)}\ll_{P_1,\dots,P_{m_1},Q_1,\dots,Q_{m_2}}&\exp_q\left(-\beta+\sum_{i=0}^{s-2}\delta_3^{(s-i)}\right)+\exp_q\left((1-\beta)\delta_3^{(s)}-\beta\delta_4^{(s)}+\sum_{i=1}^{s-2}\delta_3^{(s-i)}\right) + \\
&\sum_{j=1}^{s-2}\exp_q\left(\delta_1^{(s-j+1)}+(1-2^{1-(s-j+1)})\delta_3^{(s-j)}-2^{1-(s-j+1)}\delta_4^{(s-j)}+\sum_{i=j+1}^{s-2}\delta_3^{(s-i)}\right) + \\
&\sum_{j=0}^{s-2}\exp_q\left(-\delta_2^{(s-j)}+\sum_{i=j+1}^{s-2}\delta_3^{(s-i)}\right)+\\
&\sum_{j=0}^{s-2}\exp_q\left(\delta_1^{(s-j)}-2^{2-2(s-j)}\gamma+\sum_{i=j+1}^{s-2}\delta_3^{(s-i)}\right).
\end{align*}

It remains to choose the $\delta_k^{(\ell)}$'s so that the above bound for $b_3^{(1)}$ is smaller than a negative power of $q$. One simple choice that works is
\begin{equation}\label{deltachoice}
\delta_{k}^{(\ell)}=\begin{cases} 2^{1-2s\ell}\gamma\beta & k=1 \\ 2^{2\ell-4s^2}\gamma\beta^2 & k=2 \\ 2^{1+2\ell-4s^2}\gamma\beta^2 & k=3 \\ 2^{-2s\ell}\gamma\beta& k=4 \end{cases}.
\end{equation}
Note that our definition of $\delta_k^{(\ell)}$ depends only on $s,\beta,$ and $\gamma$, which each depend only on $P_1,\dots,P_{m_1},Q_1,\dots,Q_{m_2}$.

Clearly $\delta_2^{(\ell)}<\delta_3^{(\ell)}$ and $\delta_4^{(\ell)}<\delta_1^{(\ell)}$ for all $2\leq\ell\leq s$, and we can easily verify that each of the five exponents of $q$ appearing in our bound for $\Lambda_{P_1,\dots,P_{m_1}}^{Q_1,\dots,Q_{m_2}}(f_0,\dots,f_{m_1-1},\f{1}{2}f_{m_1}';\Psi)$ are negative. Indeed, for the argument of the first $\exp_q$ in the sum bounding $b_3^{(1)}$ from above, recalling that $s\geq 2$, we have,
\begin{align*}
-\beta+\sum_{i=0}^{s-2}\delta_3^{(s-i)} &= -\beta\left(1-\gamma\beta\sum_{i=0}^{s-2}2^{1+2(s-i)-4s^2}\right) \\
&< -\beta\left(1-2^{-4s^2+2s+2}\right) \\
&\leq -\beta\left(1-2^{-10}\right),
\end{align*}
for the argument of the second $\exp_q$, we have
\begin{align*}
(1-\beta)\delta_3^{(s)}-\beta\delta_4^{(s)}+\sum_{i=1}^{s-2}\delta_3^{(s-i)}&\leq -\gamma\beta^22^{-2s^2}\left(1-2^{1+2s-2s^2}-\sum_{i=1}^{s-2}2^{1+2(s-i)-2s^2}\right) \\
&<-\gamma\beta^2 2^{-2s^2}\left(1-2^{-3}\right),
\end{align*}
for the argument of the third $\exp_q$, we have
\[
\delta_1^{(s-j+1)}+(1-2^{1-(s-j+1)})\delta_3^{(s-j)}-2^{1-(s-j+1)}\delta_4^{(s-j)}+\sum_{i=j+1}^{s-2}\delta_3^{(s-i)}<-\gamma\beta 2^{-2s^2}
\]
for every $j=1,\dots,s-2$, for the argument of the fourth $\exp_q$, we have
\begin{align*}
-\delta_2^{(s-j)}+\sum_{i=j+1}^{s-2}\delta_3^{(s-i)} &=-\gamma\beta^22^{-4s^2+2(s-j)}\left(1-2\sum_{i=1}^{s-2-j}2^{-2i}\right) \\
&<-\gamma\beta^22^{-4s^2+4}\left(1-\f{2}{3}\right)
\end{align*}
for every $j=0,\dots,s-2$, and, finally, for the argument of the fifth $\exp_q$, we have
\[
\delta_1^{(s-j)}-2^{2-2(s-j)}\gamma+\sum_{i=j+1}^{s-2}\delta_3^{(s-i)}<-\gamma 2^{2-2s}\left(1-2^{-4}\right)
\]
for every $j=0,\dots,s-2$.

Thus, using the choice~(\ref{deltachoice}) for the $\delta_k^{(\ell)}$'s, we see that there exists $\gamma'>0$ depending only on $P_1,\dots,P_{m_1},Q_1,\dots,Q_{m_2}$ such that
\[
|\Lambda_{P_1,\dots,P_{m_1}}^{Q_1,\dots,Q_{m_2}}(F;\Psi)-\E_zf_{m_1}(z)\Lambda_{P_1,\dots,P_{m_1-1}}^{Q_1,\dots,Q_{m_2}}(f_0,\dots,f_{m_1-1};\Psi)|\ll_{P_1,\dots,P_{m_1},Q_1,\dots,Q_{m_2}}q^{-\gamma'}
\]
whenever the characteristic of $\F_q$ is sufficiently large. Since there exists, by the inductive hypothesis, $\gamma''>0$ depending only on $P_1,\dots,P_{m_1-1},Q_1,\dots,Q_{m_2}$ such that
\[
\Lambda_{P_1,\dots,P_{m_1-1}}^{Q_1,\dots,Q_{m_2}}(f_0,\dots,f_{m_1-1};\Psi)=1_{\Psi=1}\prod_{i=0}^{m_1-1}\E_zf_i(z)+O_{P_1,\dots,P_{m_1-1},Q_1,\dots,Q_{m_2}}(q^{-\gamma''})
\]
whenever the characteristic of $\F_q$ is sufficiently large, this completes the proof of the theorem.
\end{proof}

\section{Proof of Lemma~\ref{ind}}\label{lemmas}
As mentioned in the previous section, instead of using the inverse theorem for the $U^s$-norm in place of~(\ref{u2}) in the proof of Lemma~\ref{ind} and then needing to deal with nilsequences, we use the following lemma to return to the $U^2$ situation. To avoid potential confusion while reading the lemma, note that the function $F_{h_1,\dots,h_s}$ defined below is not the same as $\Delta_{h_1,\dots,h_s}F$.
\begin{lemma}\label{cs}
Let $f_1,\dots,f_m:\F_q^2\to\C$ be $1$-bounded. Set
\[
F(x):=\E_y\prod_{i=1}^mf_i(x,y)
\]
and, for every $h_1,\dots,h_t\in\F_q$, set
\[
F_{h_1,\dots,h_t}(x):=\E_y\prod_{i=1}^m\Delta_{h_1,\dots,h_t}^{(1)}f_i(x,y).
\]
Then $\|F\|_{U^s}^{2^{2s-2}}\leq \E_{h_1,\dots,h_{s-2}}\|F_{h_1,\dots,h_{s-2}}\|_{U^2}^4$ for all $s\geq 2$.
\end{lemma}
\begin{proof}[Proof of Lemma~\ref{cs}]
This is proved by repeated applications of the Cauchy--Schwarz inequality. The result is trivial when $s=2$. We will first show that $\|F\|_{U^s}^{2^{s+1}}\leq\E_h\|F_h\|_{U^{s-1}}^{2^{s-1}}$ for all $s\geq 3$. The conclusion of the lemma will then follow easily by induction.

Denote the complex conjugation map by $C:\C\to\C$, so that $Cz=\bar{z}$. By definition, we have
\begin{align*}
\|F\|_{U^s}^{2^s} &= \E_{x\in\F_q,h\in\F_q^s}\prod_{\omega\in\{0,1\}^s}C^{|\omega|}F(x+h\cdot\omega) \\
&= \E_{\substack{x\in\F_q,y\in\F_q^{\{0,1\}^s}\\h\in\F_q^s}}\prod_{\omega\in\{0,1\}^s}\prod_{i=1}^mf_i(x+h\cdot\omega,y_\omega),
\end{align*}
which, splitting the product over $\omega\in\{0,1\}^s$ up based on the value of $\omega_s$, can be written as
\[
\E_{\substack{x\in\F_q,y\in\F_q^{\{0,1\}^s} \\ h\in\F_q^{s-1}\times\{0\}}}\prod_{\substack{\omega\in\{0,1\}^s \\ \omega_s=0}}\prod_{i=1}^mC^{|\omega|}f_i(x+h\cdot\omega,y_\omega)\E_{h_s}\prod_{\substack{\omega\in\{0,1\}^s \\ \omega_s=1}}\prod_{i=1}^mC^{|\omega|}f_i(x+h\cdot\omega+h_s,y_\omega).
\]
Now we apply Cauchy--Schwarz in $x,y,$ and $h$ and use the $1$-boundedness of the $f_i$ to get that $\|F\|_{U^s}^{2^{s+1}}$ is bounded above by
\[
\E_{\substack{x\in\F_q,y\in\F_q^{\{0,1\}^s} \\ h\in\F_q^{s-1}\times\{0\}}}\E_{h_s,h_s'}\prod_{\substack{\omega\in\{0,1\}^s \\ \omega_s=1}}\prod_{i=1}^mC^{|\omega|-1}f_i(x+h\cdot\omega+h_s,y_\omega)\overline{f_i(x+h\cdot\omega+h_s',y_\omega)}.
\]
Making the change of variables $x\mapsto x-h_s',h_s\mapsto h_s+h_s'$ and noting that the value of $\prod_{\substack{\omega\in\{0,1\}^s \\ \omega_s=1}}\prod_{i=1}^mC^{|\omega|-1}\Delta_{h_s}^{(1)}f_i(x+h\cdot\omega,y_\omega)$ does not depend on $y_\omega$ whenever $\omega_s=0$, we see that the above equals
\[
\E_{h_s}\E_{\substack{x\in\F_q, y\in\F_q^{\{0,1\}^{s-1}} \\ h\in\F_q^{s-1}}}\prod_{\omega\in\{0,1\}^{s-1}}\prod_{i=1}^mC^{|\omega|}\Delta_{h_s}^{(1)}f_i(x+h\cdot\omega,y_\omega),
\]
which is exactly $\E_{h_s}\|F_{h_s}\|_{U^{s-1}}^{2^{s-1}}$.

Now, since $\Delta_h^{(1)}f$ is $1$-bounded whenever $f$ is, it follows by induction that $\|F\|_{U^s}^{2^{2s-2}}\leq \E_{h_1,\dots,h_{s-2}}\|F_{h_1,\dots,h_{s-2}}\|_{U^2}^4$ for all $s\geq 3$. Indeed, we have just shown that $\|F\|_{U^3}^{2^4}\leq\E_{h_1}\|F_{h_1}\|_{U^2}^4$, and if $\|F\|_{U^t}^{2^{2t-2}}\leq \E_{h_1,\dots,h_{t-2}}\|F_{h_1,\dots,h_{t-2}}\|_{U^2}^4$ for $t\geq 3$, then
\begin{align*}
\|F\|_{U^{t+1}}^{2^{2(t+1)-2}} &= (\|F\|_{U^{t+1}}^{2^{(t+1)+1}})^{2^{t-2}} \\
&\leq (\E_{h_1}\|F_{h_1}\|_{U^t}^{2^{t}})^{2^{t-2}} \\
&\leq \E_{h_1}\|F_{h_1}\|_{U^t}^{2^{2t-2}} \\
&\leq \E_{h_1,\dots,h_{(t+1)-2}}\|F_{h_1,\dots,h_{(t+1)-2}}\|_{U^2}^4
\end{align*}
since $(F_h)_k=F_{h,k}$ for every $h,k\in\F_q$.
\end{proof}

The proof of Lemma~\ref{ind} essentially follows the argument given at the beginning of Section~\ref{pf}, but is done in greater generality using Lemma~\ref{cs}. The argument at the beginning of Section~\ref{pf} was successful due to the linear independence of $P_1$ and $P_2$, which implied that of $P_1$ and $P_2-P_1$ as well. Similarly, the key to the following proof is the linear independence of $P_1,\dots,P_{m_1},Q_1,\dots,Q_{m_2}$, which will imply the linear independence of other related collections of polynomials $R_1,\dots,R_{m_1-1},S_1,\dots,S_{m_2+1}$ that we can apply the lemma's hypothesis to.

\begin{proof}[Proof of Lemma~\ref{ind}]
By hypothesis, there exist $c_1',c_2,\gamma'>0$ such that if $F=(f_0,\dots,f_{m_1})$ and $G=(g_0,\dots,g_{m_1-1})$ are $1$-bounded, $\Psi\in(\widehat{\F}_q)^{m_2}$, $\Phi\in(\widehat{\F}_q)^{m_2+1}$, and $\F_q$ has characteristic at least $\max(c_1',b_4)$, then
\begin{equation}\label{bus}
|\Lambda_{P_1,\dots,P_{m_1}}^{Q_1,\dots,Q_{m_2}}(F;\Psi)|\leq b_1\min_j\|f_j\|_{U^s}^{b_2}+b_3
\end{equation}
and
\begin{equation}\label{m12}
\big|\Lambda_{R_1,\dots,R_{m_1-1}}^{S_1,\dots,S_{m_2+1}}(G;\Phi)-1_{\Phi=1}\prod_{i=0}^{m_1-1}\E_xg_i(x)\big|\leq\f{c_2}{q^{\gamma'}},
\end{equation}
whenever
\[
R_i=\begin{cases} P_i-P_k & i\leq k-1 \\ P_{i+1}-P_k & i\geq k\end{cases} \text{ and } S_j=\begin{cases} Q_j & j\leq m_2 \\ P_k & j=m_2+1 \end{cases}
\]
for $k=1$ or $2$, or
\[
R_i=\begin{cases} -P_2 & i=1 \\ P_{i+1}-P_2 & i\geq 2\end{cases}  \text{ and } S_j=\begin{cases} Q_j & j\leq m_2 \\ -P_1+P_2 & j=m_2+1 \end{cases}.
\]
Indeed, note that $R_1,\dots,R_{m_1-1},S_1,\dots,S_{m_2+1}$ are linearly independent in each of these three cases since $P_1,\dots,P_{m_1},Q_1,\dots,Q_{m_2}$ are linearly independent.

Let $\delta_1,\delta_2,\delta_3,\delta_4>0$. Assume that the characteristic of $\F_q$ is at least $\max(c_1',b_4)$ and that $q^{\delta_2-\delta_3}+q^{\delta_4-\delta_1}\leq 1/2$. By Proposition~\ref{hb}, we can write
\[
f_0=f_a+f_b+f_c
\]
for some $f_a,f_b,f_c:\F_q\to\C$ with $\|f_a\|_{U^s}^*\leq q^{\delta_1}$, $\|f_b\|_{L^1}\leq q^{-\delta_2}$, $\|f_c\|_{L^\infty}\leq q^{\delta_3}$, and $\|f_c\|_{U^s}\leq q^{-\delta_4}$. Set $F_a:=(f_a,f_1,\dots,f_{m_1})$, $F_b:=(f_b,f_1,\dots,f_{m_1})$, $F_c:=(f_c,f_1,\dots,f_{m_1})$, and
\[
g(x):=\E_y\prod_{i=1}^{m_1}f_i(x+P_i(y))\prod_{j=1}^{m_2}\psi_j(Q_j(y)).
\]
Then
\[
\Lambda_{P_1,\dots,P_{m_1}}^{Q_1,\dots,Q_{m_2}}(F;\Psi)=\Lambda_{P_1,\dots,P_{m_1}}^{Q_1,\dots,Q_{m_2}}(F_a;\Psi)+\Lambda_{P_1,\dots,P_{m_1}}^{Q_1,\dots,Q_{m_2}}(F_b;\Psi)+\Lambda_{P_1,\dots,P_{m_1}}^{Q_1,\dots,Q_{m_2}}(F_c;\Psi)
\]
and $\Lambda_{P_1,\dots,P_{m_1}}^{Q_1,\dots,Q_{m_2}}(F_a;\Psi)=\E_xf_a(x)g(x)$.

As in the discussion at the beginning of Section~\ref{pf}, the term $\Lambda_{P_1,\dots,P_{m_1}}^{Q_1,\dots,Q_{m_2}}(F_b;\Psi)$ is easy to bound using the triangle inequality and the fact that $f_1,\dots,f_{m_1}$ are all $1$-bounded. Indeed, 
\[
\Lambda_{P_1,\dots,P_{m_1}}^{Q_1,\dots,Q_{m_2}}(F_b;\Psi)\leq \E_x|f_b(x)|\prod_{i=1}^{m_1}|f_i(x+P_i(y))|\prod_{j=1}^{m_2}|\psi_j(Q_j(y))|\leq \|f_b\|_{L^1}\leq q^{-\delta_2}.
\]

To bound $|\Lambda_{P_1,\dots,P_{m_1}}^{Q_1,\dots,Q_{m_2}}(F_c;\Psi)|$, note that $q^{-\delta_3}f_c,f_1,\dots,f_{m_1}$ are all $1$-bounded, so that by~(\ref{bus}) we have
\begin{align*}
|\Lambda_{P_1,\dots,P_{m_1}}^{Q_1,\dots,Q_{m_2}}(F_c;\Psi)| &= q^{\delta_3}|\Lambda_{P_1,\dots,P_{m_1}}^{Q_1,\dots,Q_{m_2}}(q^{-\delta_3}f_c,f_1,\dots,f_{m_1};\Psi)| \\ 
&\leq q^{\delta_3}(b_1\|q^{-\delta_3}f_c\|_{U^s}^{b_2}+b_3) \\
&\leq q^{(1-b_2)\delta_3-b_2\delta_4}b_1+q^{\delta_3}b_3,
\end{align*}
using the bound $\|f_c\|_{U^s}\leq q^{-\delta_4}$.

Finally, to bound $\Lambda_{P_1,\dots,P_{m_1}}^{Q_1,\dots,Q_{m_2}}(F_a;\Psi)$, we use that $\Lambda_{P_1,\dots,P_{m_1}}^{Q_1,\dots,Q_{m_2}}(F_a;\Psi)=\langle f_a,\overline{g}\rangle$ and
\[
|\langle f_a,g\rangle|\leq\|f_a\|_{U^s}^*\|g\|_{U^s}\leq q^{\delta_1}\|g\|_{U^s},
\]
so that it remains to bound $\|g\|_{U^s}$.

Since $f_1,\dots,f_{m_1},\psi_1,\dots,\psi_{m_2}$ are all $1$-bounded, we have by Lemma~\ref{cs} that
\[
\|g\|_{U^s}^{2^{2s-2}}\leq\E_{h_1,\dots,h_{s-2}}\|g_{h_1,\dots,h_{s-2}}\|_{U^2}^4.
\]
Let $\phi_{m_2+1}\in\widehat{\F}_q$ and, for each $j=1,\dots,m_2$, set $\phi_j=\psi_j$ if $s=2$ and $\phi_j=1$ otherwise. Then
\[
g_{h_1,\dots,h_{s-2}}(x)=\E_y\prod_{i=1}^{m_1}\Delta_{h_1,\dots,h_{s-2}}f_i(x+P_i(y))\prod_{j=1}^{m_2}\phi_j(Q_j(y))
\]
and, for each $k=1,\dots,m_1$, we have that $\widehat{g_{h_1,\dots,h_{s-2}}}(\phi_{m_2+1})$ equals
\begin{align*}
&\E_{x,y}\overline{\phi_{m_2+1}(x)}\prod_{i=1}^{m_1}\Delta_{h_1,\dots,h_{s-2}}f_i(x+P_i(y))\prod_{j=1}^{m_2}\phi_j(Q_j(y)) \\
&= \E_{x,y}\overline{\phi_{m_2+1}(x-P_k(y))}\prod_{i=1}^{m_1}\Delta_{h_1,\dots,h_{s-2}}f_i(x+P_i(y)-P_k(y))\prod_{j=1}^{m_2}\phi_j(Q_j(y)) \\
&= \Lambda_{R_1,\dots,R_{m_1-1}}^{S_1,\dots,S_{m_2+1}}(g_0,\dots,g_{m_1-1};\phi_1,\dots,\phi_{m_2+1}),
\end{align*}
where
\[
R_i=\begin{cases} P_i-P_k & i\leq k-1 \\ P_{i+1}-P_k & i\geq k\end{cases}, S_j=\begin{cases} Q_j & j\leq m_2 \\ P_k & j=m_2+1 \end{cases},
\]
and
\[
g_i=\begin{cases} \overline{\phi_{m_2+1}}\Delta_{h_1,\dots,h_{s-2}}f_k & i=0 \\ \Delta_{h_1,\dots,h_{s-2}}f_i & i\leq k-1 \\ \Delta_{h_1,\dots,h_{s-2}}f_{i+1} & i\geq k \end{cases}.
\]

Note that the $g_i$'s are $1$-bounded since $\phi_{m_2+1}$ and $f_1,\dots,f_{m_2}$ are $1$-bounded. We can thus apply the estimate~(\ref{m12}) when $k=1$ and $2$ to get that
\[
|\widehat{g_{h_1,\dots,h_{s-2}}}(\phi_{m_2+1})-1_{\Phi=1}\prod_{i=0}^{m_1-1}\E_zg_i(z)|\leq\f{c_2}{q^{\gamma'}}
\]
for every $h_1,\dots,h_{s-2}\in\F_q$. When $k=1$, the above estimate tells us that 
\[
|\widehat{g_{h_1,\dots,h_{s-2}}}(\phi_{m_2+1})|\leq \min_{i\geq 2}|\E_z\Delta_{h_1,\dots,h_{s-2}}f_i(z)|+\f{c_2}{q^{\gamma'}}
\]
since each $g_i$ is $1$-bounded and $g_i=\Delta_{h_1,\dots,h_{s-2}}f_{i+1}$ for $i\geq 1$. Similarly, when $k=2$, the above also tells us that
\[
|\widehat{g_{h_1,\dots,h_{s-2}}}(\phi_{m_2+1})|\leq |\E_z\Delta_{h_1,\dots,h_{s-2}}f_1(z)|+\f{c_2}{q^{\gamma'}}.
\]
Thus, since $\|g_{h_1,\dots,h_{s-2}}\|_{U^2}^4\leq\max_{\phi\in\widehat{\F}_q}|\widehat{g_{h_1,\dots,h_{s-2}}}(\phi)|^2$ by~(\ref{u2}) and the $1$-boundedness of $g_{h_1,\dots,h_{s-2}}$, we have
\begin{align*}
\|g\|_{U^s}^{2^{2s-2}} &\leq \E_{h_1,\dots,h_{s-2}}\max_{\phi\in\widehat{\F}_q}|\widehat{g_{h_1,\dots,h_{s-2}}}(\phi)|^2 \\
&\leq \E_{h_1,\dots,h_{s-2}}\min_{i\geq 1}\|\Delta_{h_1,\dots,h_{s-2}}f_i\|_{U^1}^2+\f{2c_2+c_2^2}{q^{\gamma'}} \\
&=\min_{i\geq 1}\E_{h_1,\dots,h_{s-2}}\E_{x,h_{s-1}}\Delta_{h_1,\dots,h_{s-2}}f_i(x+h_{s-1})\overline{\Delta_{h_1,\dots,h_{s-2}}f_i(x)}+\f{2c_2+c_2^2}{q^{\gamma'}} \\
&=\min_{i\geq 1}\E_{x,h_1,\dots,h_{s-1}}\Delta_{h_1,\dots,h_{s-1}}f_i(x)+\f{2c_2+c_2^2}{q^{\gamma'}} \\
&=\min_{i\geq 1}\|f_i\|_{U^{s-1}}^{2^{s-1}}+\f{2c_2+c_2^2}{q^{\gamma'}}.
\end{align*}
Using the fact that $x_1^{1/2}+x_2^{1/2}\geq (x_1+x_2)^{1/2}$ whenever $x_1,x_2>0$, the above implies that
\[
\|g\|_{U^s}\leq\min_{i\geq 1}\|f_i\|_{U^{s-1}}^{2^{1-s}}+\left(\f{2c_2+c_2^2}{q^{\gamma'}}\right)^{2^{2-2s}}.
\]

Setting $c_2':=2c_2+c_2^2$, we conclude that
\[
|\Lambda_{P_1,\dots,P_{m_1}}^{Q_1,\dots,Q_{m_2}}(F;\Psi)|\leq q^{\delta_1}\min_{i\geq 1}\|f_i\|_{U^{s-1}}^{2^{1-s}}+q^{\delta_1}\left(\f{c_2'}{q^{\gamma'}}\right)^{2^{2-2s}}+q^{-\delta_2}+q^{(1-b_2)\delta_3-b_2\delta_4}b_1+q^{\delta_3}b_3.
\]

To replace the $\min_{i\geq 1}$ with a $\min_{i\geq 0}$ here, we note that the same argument can be run by decomposing $f_1$ instead of $f_0$ using Proposition~\ref{hb}. Then the estimate~(\ref{m12}) for the third choice of $R_i$'s and $S_j$'s becomes relevant, and by the same argument one shows that $|\Lambda_{P_1,\dots,P_{m_1}}^{Q_1,\dots,Q_{m_2}}(F;\Psi)|$ is bounded above by
\[
q^{\delta_1}\|f_0\|_{U^{s-1}}^{2^{1-s}}+q^{\delta_1}\left(\f{c_2'}{q^{\gamma'}}\right)^{2^{2-2s}}+q^{-\delta_2}+q^{(1-b_2)\delta_3-b_2\delta_4}b_1+q^{\delta_3}b_3
\]
as well.
\end{proof}

\bibliographystyle{plain}
\bibliography{bib}

\end{document}